\documentclass{kms-j}

\issueinfo{}
  {}
  {}
  {}
\pagespan{1}{}
\copyrightinfo{}
  {Korean Mathematical Society}

\usepackage{graphicx, amsmath, enumitem}
\allowdisplaybreaks

\theoremstyle{plain}
\newtheorem{theorem}{Theorem}[section]
\newtheorem{proposition}[theorem]{Proposition}
\newtheorem{lemma}[theorem]{Lemma}

\theoremstyle{definition}
\newtheorem*{definition}{Definition}
\newtheorem{example}[theorem]{Example}

\theoremstyle{remark}
\newtheorem*{remark}{Remark}

\newcommand{\g}{$\mathfrak{g}$}

\begin{document}

\title[Odd nilpotent element and $\mathfrak{osp}(1|2)$-subalgebra in $\mathfrak{gl}(m|n)$]
{Odd nilpotent element and $\mathfrak{osp}(1|2)$-subalgebra in $\mathfrak{gl}(m|n)$}

\author[J. Ko]{Junseo Ko}
\address{Junseo Ko \\ Department of Mathematical Sciences \\ Seoul National University \\ Seoul 08826, Korea}
\email{kojs5725@snu.ac.kr}



\begin{abstract}
In this paper, we investigate the conditions under which an odd nilpotent element in $\mathfrak{gl}(m|n)$ lies inside an $\mathfrak{osp}(1|2)$-subalgebra. In the case of the classical Lie algebra $\mathfrak{gl}_m$, every nilpotent element can be embedded into an $\mathfrak{sl}_2$-subalgebra, which is the result of the Jacobson-Morozov Theorem. In the case of the Lie superalgebra $\mathfrak{gl}(m|n)$, we define super Jordan matrices and prove that an odd nilpotent element $e$ is contained in an $\mathfrak{osp}(1|2)$-subalgebra if and only if $e$ lies in the orbit of a super Jordan matrix consisting only of super Jordan blocks of odd size.
\end{abstract}

\maketitle


\section{Introduction}
In this paper, we study odd nilpotent elements and their associated $\mathfrak{osp}(1|2)$-subalgebras in $\mathfrak{gl}(m|n)$. Entova-Aizenbud and Serganova \cite{ES} established, via a categorical approach, a necessary and sufficient condition for an odd nilpotent element $e \in \text{Lie}(G)$\footnote{Here, $G$ denotes a quasi-reductive algebraic supergroup.} to induce an embedding of supergroups $\text{OSp}(1|2) \hookrightarrow G$ such that $e$ lies in the image of the corresponding Lie algebra homomorphism. Building on their result, we provide a constructive characterization of odd nilpotent elements $e \in \mathfrak{gl}(m|n)$ that can be contained in an $\mathfrak{osp}(1|2)$-subalgebra. Analogous to the one-to-one correspondence between nilpotent elements in $\mathfrak{sl}_n$ and partitions of $n$, we construct a combinatorial object that provides a one-to-one correspondence with nilpotent elements in $\mathfrak{gl}(m|n)$ that can be embedded into an $\mathfrak{osp}(1|2)$-subalgebra. \\
\indent In Section 2.1, we recall the definition of Lie superalgebras and review representative examples, including the general linear, special linear, and ortho-symplectic Lie superalgebras \cite{CW, Kac}. In Section 2.2, we introduce the theorems of Jacobson–Morozov \cite{Jac, Mor} and Kostant \cite{Ko}, which establish a one-to-one correspondence between the set of $\mathfrak{sl}_2$-triples and the set of nilpotent orbits in a semisimple Lie algebra $\mathfrak{g}$. We also describe the representatives of nilpotent orbits in $\mathfrak{sl}_n$ and construct explicit $\mathfrak{sl}_2$-triples containing a given nilpotent element in $\mathfrak{sl}_n$ \cite{CMc}.\\
\indent In section 3.1, we describe representatives of odd nilpotent orbits in $\mathfrak{gl}(m|n)$ \cite{JN}. In section 3.2, given an odd nilpotent element $e \in \mathfrak{gl}(m|n)$, we construct an explicit $\mathfrak{sl}_2$-triple containing the even nilpotent element $E = e^2$.\\
\indent In Section 4, we introduce the notion of a super Jordan matrix (or block) and prove the main theorem of this paper. A super Jordan matrix is the natural super-analogue of the Jordan canonical form, obtained by a change of basis with alternating parities. The main theorem is stated as follows.

\begin{theorem}[Theorem \ref{thm;theorem1}]
    Let $\mathfrak{g} = \mathfrak{gl}(m|n)$, and $\text{Lie}(G_{\bar{0}}) = \mathfrak{g}_{\bar{0}}$ for some semisimple, simply connected group $G_{\bar{0}}$. Then an odd nilpotent element $e \in \mathfrak{g}$ is contained in an $\mathfrak{osp}(1|2)$-subalgebra of $\mathfrak{g}$ if and only if $e$ lies in the $G_{\bar{0}}$-orbit of a super Jordan matrix consisting entirely of super Jordan blocks of odd size.
\end{theorem}

\noindent To this end, we establish an equivalent necessary and sufficient condition for an odd nilpotent element $e \in \mathfrak{g}$ to be contained in an $\mathfrak{osp}(1|2)$-subalgebra.

\section{Preliminary}
\subsection{Lie superalgebra}

We review the notion of a Lie superalgebra and discuss important examples, including the general linear, special linear, and ortho-symplectic Lie superalgebras \cite{CW, Kac}. Throughout this paper, Lie superalgebras are defined over $\mathbb{C}$.

\vspace{1pt}

\begin{definition}
     A \textbf{vector superspace} $V$ is a vector space endowed with a $\mathbb{Z}_2$-gradation: $V = V_0 \oplus V_1$. The \textbf{dimension} of the vector superspace $V$ is the tuple $\dim V =(\dim V_0 | \dim V_1)$. The \textbf{parity} of a homogeneous element $a \in V_i$ is denoted by $|a| = i, \ i \in \mathbb{Z}_2$. An element in $V_0$ is called \textbf{even}, while an element in $V_1$ is called \textbf{odd}. A \textbf{subspace} of a vector superspace $V = V_0 \oplus V_1$ is a vector superspace $W = W_0 \oplus W_1$ with compatible $\mathbb{Z}_2$-gradation. The superspace with even subspace $\mathbb{C}^m$ and odd subspace $\mathbb{C}^n$ is denoted by $\mathbb{C}^{m|n}$.\\
     \indent A \textbf{superalgebra} $A$ is a vector superspace $A = A_0 \oplus A_1$ equipped with bilinear multiplication satisfying $A_iA_j \subseteq A_{i+j}$, for $i, j \in \mathbb{Z}_2$. \textbf{Subalgebras} and \textbf{ideals} of superalgebras are also understood in the $\mathbb{Z}_2$-graded sense. 
\end{definition}

\begin{definition}
    A \textbf{Lie superalgebra} is a superalgebra $\mathfrak{g} = \mathfrak{g}_{\bar{0}} \oplus \mathfrak{g}_{\bar{1}}$ with bilinear bracket operation $[\cdot, \cdot]$ satisfying the following two axioms: for homogeneous elements $a, b, c \in \mathfrak{g},$
    
    \begin{flushleft}\setlength{\leftskip}{10pt}
    (Skew-supersymmetry) \quad $[a, b] = -(-1)^{|a||b|}[b, a]$.\\
    (Super Jacobi identity) \quad $[a,[b,c]] = [[a,b],c] + (-1)^{|a||b|}[b,[a,c]]$.
    \end{flushleft}
    Let $\mathfrak{g}$ and $\mathfrak{g}'$ be Lie superalgebras. A \textbf{homomorphism} of Lie superalgebras is an even linear map $f : \mathfrak{g} \rightarrow \mathfrak{g'}$ that preserves the bracket operation.
\end{definition}

\begin{example}
    Let $V = V_{\bar{0}} \oplus V_{\bar{1}}$ be a vector superspace so that $\text{End}(V)$ is an associative superalgebra. Now, define the bracket operation as follows:
    \[ [a,b] = ab - (-1)^{|a||b|}ba, \quad \text{homogenous} \ a,b \in \text{End}(V) \]
    A Lie superalgebra $\text{End}(V)$ equipped with supercommutator $[\cdot, \cdot]$ is called the \textbf{general linear Lie superalgebra} and denoted by $\mathfrak{gl}(V)$. When $V = \mathbb{C}^{m|n}$ we write $\mathfrak{gl}(m|n)$ for $\mathfrak{gl}(V)$.\\
    \indent We may parametrize an ordered basis of $V$ by the set
    \[ I(m|n) = \{\bar{1}, \ \ldots \ , \ \bar{m} \ ; \ 1, \ \ldots \ , \ n\}. \]
    With respect to this ordered basis, $\mathfrak{gl}(V)$ can be realized as
    \begin{equation}\label{eqn;eqn1}
    \small
    \mathfrak{gl}(V) = \left\{ \begin{pmatrix}
    a & b \\
    c & d \\    
    \end{pmatrix} : a \in \text{Mat}_{m\times m}, \ b \in \text{Mat}_{m\times n}, \ c \in \text{Mat}_{n\times m}, \ d \in \text{Mat}_{n\times n} \right\}.
    \end{equation}
    In particular,
    \begin{align*}
    \mathfrak{gl}(V)_{\bar{0}} &= \left\{ \begin{pmatrix}
    a & 0 \\
    0 & d \\    
    \end{pmatrix} : a \in \text{Mat}_{m\times m}, \ d \in \text{Mat}_{n\times n} \right\}, \\
    \mathfrak{gl}(V)_{\bar{1}} &= \left\{ \begin{pmatrix}
    0 & b \\
    c & 0 \\    
    \end{pmatrix} : b \in \text{Mat}_{m \times n}, \ c \in \text{Mat}_{n \times m} \right\}.
    \end{align*}
\end{example}

\vspace{1pt}

\begin{example}
    For an element $g \in \mathfrak{gl}(m|n)$ of the form (\ref{eqn;eqn1}), we define the \textbf{supertrace} as
    \[ \text{str}(g) = \text{tr}(a) - \text{tr}(d). \]
    Then, the subspace
    \[ \mathfrak{sl}(m|n) = \{g \in \mathfrak{gl}(m|n) : \text{str}(g) = 0 \} \]
    is a subalgebra of $\mathfrak{gl}(m|n)$, and it is called the \textbf{special linear Lie superalgebra}.
\end{example}

\vspace{1pt}

\begin{example}
    Let $V = V_{\bar{0}} \oplus V_{\bar{1}}$ be a vector superspace with $\dim V = (m|n)$. Let $F$ be a non-degenerate even supersymmetric bilinear form on $V$, so that $V_{\bar{0}}$, and $V_{\bar{1}}$ are orthogonal and the restriction of $F$ to $V_{\bar{0}}$ is a symmetric and to $V_{\bar{1}}$ a skew-symmetric form (in particular, $n = 2r$ is even).\\
    \indent Now, define $\mathfrak{osp}(V) = \mathfrak{osp}(V)_{\bar{0}} \oplus \mathfrak{osp}(V)_{\bar{1}}$, where
    \[ \mathfrak{osp}(V)_s = \{ g \in \mathfrak{gl}(V)_s : F(g(x), y) = -(-1)^{s|x|}F(x, g(y)), \ \forall x,y \in V \}, \ s \in \mathbb{Z}_2.\]
    Here, $\mathfrak{osp}(V)$ is called the \textbf{ortho-symplectic Lie superalgebra}. In other words, $\mathfrak{osp}(V)$ is the subalgebra of $\mathfrak{gl}(V)$ that preserves a non-degenerate supersymmetric bilinear form. When $V = \mathbb{C}^{m|n}$ we write $\mathfrak{osp}(m|n)$ for $\mathfrak{osp}(V)$.\\
    \indent If $m = 2l+1$, a matrix in $\mathfrak{osp}(m|n)$ is of the form
    \begin{equation}\label{eqn;eqn2}
    \left(
    \begin{array}{c|c}
    \begin{matrix} a & b & u\\ c & -a^t & v \\ -v^t & -u^t & 0 \end{matrix} & \begin{matrix} x & x_1 \\ y & y_1  \\ z & z_1 \end{matrix} \\
    \hline
    \begin{matrix} y_1^t & x_1^t & z_1^t \\ -y^t & -x^t & -z^t \end{matrix} & \begin{matrix} d & e \\ f & -d^t \end{matrix}
    \end{array}
    \right) \in \text{Mat}_{(2l+2r+1)\times (2l+2r+1)},
    \end{equation}
    where $a$ is any $(l \times l)$-matrix; $b$ and $c$ are skew-symmetric $(l \times l)$-matrices; $d$ is any $(r \times r)$-matrix; $e$ and $f$ are symmetric $(r \times r)$-matrices; $u$ and $v$ are $(l \times 1)$-matrices; $x, x_1, y$ and $y_1$ are $(l \times r)$-matrices; $z$ and $z_1$ are $(1 \times r)$-matrices.\\
    \indent If $m = 2l$, a matrix in $\mathfrak{osp}(m|n)$ has the same form as in (\ref{eqn;eqn2}), with the middle row and column deleted.
\end{example}

\vspace{1pt}

\subsection{Nilpotent orbits in semisimple Lie algebra}
We introduce the Jacobson–Morozov theorem \cite{Jac, Mor} for a semisimple Lie algebra \g, which states that every nilpotent element $e \in \mathfrak{g}$ is contained in an $\mathfrak{sl}_2$-subalgebra. We then discuss Kostant’s theorem \cite{Ko}, which establishes the uniqueness of such an $\mathfrak{sl}_2$-subalgebra. Furthermore, we explicitly construct an $\mathfrak{sl}_2$-subalgebra in the case $\mathfrak{g} = \mathfrak{sl}_n$. This section follows the exposition in \cite{CMc}.

\begin{definition}
    Let \g\ be a semisimple Lie algebra. An \textbf{$\mathfrak{sl}_2$-triple} in \g\ is a triple of elements $\{h, e, f\} \subseteq  \mathfrak{g}$ satisfying the following conditions:
    \[ [e,f] = h,\; [h,e] = 2e,\; [h,f] = -2f. \]
    $h, e, f$ are called the semisimple, nilpositive and nilnegative elements of the triple, respectively.
\end{definition}
\begin{remark}
    A subalgebra generated by the $\mathfrak{sl}_2$-triple is isomorphic to $\mathfrak{sl}_2$.
\end{remark}

\vspace{1pt}

We define a bijection
\[\Omega : \{\text{orbits of $\mathfrak{sl}_2$-triples in } \mathfrak{g}\} \; \rightarrow \; \{\text{nilpotent orbits in } \mathfrak{g} \}, \]
where the surjectivity and injectivity of $\Omega$ follow from the Jacobson–Morozov theorem and Kostant’s theorem, respectively.

\begin{theorem}[Jacobson-Morozov]
    Let \g\ be a semisimple Lie algebra and $e\in \mathfrak{g}$ a non-zero nilpotent element. Then, there exists an $\mathfrak{sl}_2$-triple in \g\ whose nilpositive element is $e$.
\end{theorem}

\begin{theorem}[Kostant]
    Let \g\ be a semisimple algebra and $\{h,e,f\}, \{h',e,f'\}$ be two $\mathfrak{sl}_2$-triples in \g\ with the same nilpositive element. Then, there exists $z \in u^e := \mathfrak{g}^e \cap [\mathfrak{g}, e]$ \footnote{Define $\mathfrak{g}^{e} = \{x \in \mathfrak{g} \ | \ [e,x] = 0 \}$ and $[\mathfrak{g}, e] = \{[w,e] \ | \ w \in \mathfrak{g}\}.$} such that $Exp\ z(h) = h',\ Exp\ z(e) = e,\ Exp\ z(f) = f'$.
\end{theorem}

We now construct an explicit $\mathfrak{sl}_2$-triple containing a given nilpotent element in $\mathfrak{sl}_n$. Recall that nilpotent elements in $\mathfrak{sl}_n$ can be expressed as follows. 

\begin{theorem}
    Let $\mathfrak{g} = \mathfrak{sl}_n$. There is a one-to-one correspondence between the set of nilpotent orbits in $\mathfrak{g}$ and the set of partitions of $n$. Precisely, let $J_i \in \text{Mat}_{i\times i}$ be a Jordan block of the form
\[J_i = \begin{pmatrix}
0 & 1 & \cdots & 0 & 0\\
0 & 0 & \cdots & 0 & 0\\
\vdots & \vdots & \ddots & \vdots & \vdots\\
0 & 0 & \ldots & 0 & 1\\
0 & 0 & \ldots & 0 & 0\\
\end{pmatrix}.  \]
Then, any nilpotent $E \in \mathfrak{sl}_n$ is similar to a Jordan canonical form $E_{[d_1, d_2 \ldots, d_k]}$ as below.\\
\begin{equation}\label{eqn;eqn3}
    E_{[d_1, d_2 \ldots, d_k]} = \begin{pmatrix}
    J_{d_1} & 0 & \cdots & 0\\
    0 & J_{d_2} & \cdots & 0\\
    \vdots & \vdots & \ddots & \vdots\\
    0 & 0 & \ldots & J_{d_k}\\
\end{pmatrix},
\end{equation}
where $[d_1, d_2, \ldots, d_k]$ is a partition of $n$.
\end{theorem}

\begin{proposition}\label{prop;proposition1}
    Given a nilpotent element $E_{[d_1, d_2 \ldots, d_k]} \in \mathfrak{sl}_n$ of the form (\ref{eqn;eqn3}), the semisimple element $H_{[d_1, d_2 \ldots, d_k]}$ and the nilnegative element $F_{[d_1, d_2 \ldots, d_k]}$ described below form the associated $\mathfrak{sl}_2$-triple.
    \[
    H_{[d_1, d_2 \ldots, d_k]} = \begin{pmatrix}
    D_{d_1} & 0 & \cdots & 0\\
    0 & D_{d_2} & \cdots & 0\\
    \vdots & \vdots & \ddots & \vdots\\
    0 & 0 & \ldots & D_{d_k}\\
\end{pmatrix}, \
    F_{[d_1, d_2 \ldots, d_k]} = \begin{pmatrix}
    \tilde{J}_{d_1} & 0 & \cdots & 0\\
    0 & \tilde{J}_{d_2} & \cdots & 0\\
    \vdots & \vdots & \ddots & \vdots\\
    0 & 0 & \ldots & \tilde{J}_{d_k}\\
\end{pmatrix},
    \]
    where each $D_{d_i}$ and $\tilde{J}_{d_i}$ are defined as follows for $i = 1, \ldots, k$:
    \[
    D_{d_i} = \begin{pmatrix}
    d_i - 1 & 0 & 0 & \cdots & 0 & 0\\
    0 & d_i - 3 & 0 & \cdots & 0 & 0\\
    0 & 0 & d_i - 5 & \cdots & 0 & 0\\
    \vdots & \vdots & \vdots & \ddots & \vdots & \vdots\\
    0 & 0 & 0 & \ldots & -d_i + 3 & 0\\
    0 & 0 & 0 & \ldots & 0 & -d_i + 1\\
\end{pmatrix},
    \]
    \[
    \tilde{J}_{d_i} = \begin{pmatrix}
    0 & 0 & 0 & \cdots & 0 & 0\\
    b_1^i & 0 & 0 & \cdots & 0 & 0\\
    0 & b_2^i & 0 & \cdots & 0 & 0\\
    \vdots & \vdots & \vdots & \ddots & \vdots & \vdots\\
    0 & 0 & 0 & \ldots & 0 & 0\\
    0 & 0 & 0 & \ldots & b_{d_i-1}^i & 0\\
\end{pmatrix},
    \]
    where each $b_j^i$ is defined inductively as $b_1^i = d_i - 1$ and $b_j^i - b_{j-1}^i = (D_{d_i})_{jj}$, $j = 1, \ldots, d_i - 1$.
\end{proposition}

\begin{proof}
    We simply write $E = E_{[d_1, d_2 \ldots, d_k]}$, $F = F_{[d_1, d_2 \ldots, d_k]}$ and $H = H_{[d_1, d_2 \ldots, d_k]}$.
    To verify that these elements form an $\mathfrak{sl}_2$-triple, it suffices to show that 
    \[ [E,F] = H, \quad [H,E] = 2E, \quad [H,F] = -2F. \]
    First, we have $[H,E] = 2E$ by the result of \cite{CMc}. In addition,
    \begin{equation*}
        [E,F] = \begin{pmatrix}
    J_{d_1}\tilde{J}_{d_1} - \tilde{J}_{d_1}J_{d_1} & 0 & \cdots & 0\\
    0 & J_{d_2}\tilde{J}_{d_2} - \tilde{J}_{d_2}J_{d_2} & \cdots & 0\\
    \vdots & \vdots & \ddots & \vdots\\
    0 & 0 & \ldots & J_{d_k}\tilde{J}_{d_k} - \tilde{J}_{d_k}J_{d_k}\\
\end{pmatrix}.
    \end{equation*}
    Since each block satisfies
    \begin{equation*}
        J_{d_i}\tilde{J}_{d_i} - \tilde{J}_{d_i}J_{d_i} = \begin{pmatrix}
d_i - 1 & 0 & \cdots & 0 & 0\\
0 & (D_{d_i})_{22} & \cdots & 0 & 0\\
\vdots & \vdots & \ddots & \vdots & \vdots\\
0 & 0 & \ldots & (D_{d_i})_{(d_i-1)(d_i-1)} & 0\\
0 & 0 & \ldots & 0 & -d_i + 1\\
\end{pmatrix} = D_{d_i},
    \end{equation*}
    we obtain $[E,F] = H$. Similarly,
    \begin{equation*}
        [H,F] = \begin{pmatrix}
    D_{d_1}\tilde{J}_{d_1} - \tilde{J}_{d_1}D_{d_1} & 0 & \cdots & 0\\
    0 & D_{d_2}\tilde{J}_{d_2} - \tilde{J}_{d_2}D_{d_2} & \cdots & 0\\
    \vdots & \vdots & \ddots & \vdots\\
    0 & 0 & \ldots & D_{d_k}\tilde{J}_{d_k} - \tilde{J}_{d_k}D_{d_k}\\
\end{pmatrix}.
    \end{equation*}
    Since each block satisfies
    \begin{equation*}
        D_{d_i}\tilde{J}_{d_i} - \tilde{J}_{d_i}D_{d_i}
        = -2\tilde{J}_{d_i},
    \end{equation*}
    it follows that $[H,F] = -2F$. This completes the verification that $\{H, E, F\}$ forms an $\mathfrak{sl}_2$-triple.
\end{proof}

\begin{example}
    Let $e = \begin{pmatrix}
        0 & 1 & 0\\
        0 & 0 & 1\\
        0 & 0 & 0
    \end{pmatrix}$ be a nilpotent element in $\mathfrak{sl}_3$. Then, 
    \[
    \left\{
    e = \begin{pmatrix}
        0 & 1 & 0\\
        0 & 0 & 1\\
        0 & 0 & 0
    \end{pmatrix}, \ 
    h = \begin{pmatrix}
        2 & 0 & 0\\
        0 & 0 & 0\\
        0 & 0 & -2
    \end{pmatrix}, \ 
    f = \begin{pmatrix}
        0 & 0 & 0\\
        2 & 0 & 0\\
        0 & 2 & 0
    \end{pmatrix}
    \right\}
    \]
    forms an $\mathfrak{sl}_2$-triple containing $e$.
\end{example}

\vspace{1pt}

\section{Odd nilpotent orbits in $\mathfrak{gl}(m|n)$}

In this section, we present the classification of odd nilpotent orbits in $\mathfrak{gl}(m|n)$, following the work of L.A. Jenkins and D.K. Nakano \cite{JN}. For an odd nilpotent element $e \in \mathfrak{gl}(m|n)$, we define an even nilpotent element $E = e^2$ and construct an associated $\mathfrak{sl}_2$-triple $\{E, F, H\}$ using Proposition \ref{prop;proposition1}.

\vspace{1pt}

\subsection{Matrix realization of odd nilpotent orbit representatives in $\mathfrak{gl}(m|n)$}

Let $\mathfrak{g} = \mathfrak{g}_{\bar{0}} \oplus \mathfrak{g}_{\bar{1}} =  \mathfrak{gl}(m|n)$, and let $G_{\bar{0}}$ be the corresponding semisimple simply connected group, such that $\text{Lie}(G_{\bar{0}}) = \mathfrak{g}_{\bar{0}}$. The space $\mathfrak{g}_{\bar{1}}$ is regarded as a $G_{\bar{0}}$-module via the adjoint action. For simplicity, blank entries in matrices are understood to be zeros.

\begin{theorem}[Jenkins, Nakano \cite{JN}]
    Let $\mathfrak{g} = \mathfrak{gl}(m|n)$. The complete set of $G_{\bar{0}}$-orbit representatives of odd nilpotent elements in $\mathfrak{g}$ is given by matrices
    \begin{equation}\label{eqn;eqn4}
    e = 
    \left(
    \begin{array}{c|c}
    0 & e^+ \\
    \hline
    e^- & 0
    \end{array}
    \right),
    \end{equation}
    where
    \[
    e^+ = 
    \left(
    \begin{array}{c|c}
    I_r & 0 \\
    \hline
    0 & 0
    \end{array}
    \right)
    \quad
    \text{and}
    \quad
    e^- = 
    \left(
    \begin{array}{ccc|ccc}
    J_1 & & & | & 0 & 0\\
     & \ddots & & C_{r_1} & 0 & 0\\
     & & J_t & | & 0 & 0\\
    \hline
    - & R_{r_2} & - & 0 & 0 & 0\\
    0 & 0 & 0 & 0 & I_s & 0\\
    0 & 0 & 0 & 0 & 0 & 0
    \end{array}
    \right).
    \]
    Here, $I_r$ (resp. $I_s$) is a $r \times r$ (resp. $s \times s$) identity matrix, $J_1, \ldots, J_t$ are Jordan blocks with zero eigenvalues, where the Jordan block $J_i$ is of size $k_i \times  k_i$, $i = 1, 2, \ldots, t$ with $k_1 \geq k_2 \geq \cdots \geq k_t$ with $r = \sum_{i = 1}^t k_i$. Furthermore, the matrix $C_{r_1}$ (resp. $R_{r_2}$) are of size $r \times r_1$ (resp. $r_2 \times r$) and of the form
    \[
    C_{r_1} = (e_{i_1} \ e_{i_2} \ \cdots \ e_{i_{r_1}}), \quad R_{r_2} = (e_{j_1} \ e_{j_2} \ \cdots \ e_{j_{r_2}})^t,
    \]
    where $e_k$ is the column vector with a single $1$ in the $k$-th row and zeros elsewhere. Each index $i_p$ (resp. $j_q$) belongs to the set $\{k_1, k_1 + k_2, \ldots, k_1 + k_2 + \cdots + k_t \}$ (resp. $\{1, 1 + k_1, \ldots, 1 + k_1 + \cdots + k_{t-1} \}$).
\end{theorem}

\vspace{1pt}

\subsection{Construction of $\mathfrak{sl}_2$-triple of $E$}
Now, we define $E = e^2$. Assume that an odd nilpotent $e$ is of the form (\ref{eqn;eqn4}). Then, $E$ is of the form

\begin{equation}\label{eqn;eqn5}
E=
\left(
    \begin{array}{c|c}
    E_1 &  \\
    \hline
     & E_2
    \end{array}
\right)
    =
\left(
\begin{array}{c|c}
    \begin{array}{ccc|cc}
        J_1 & & & | &\\
        & \ddots & & C_{r_1} & 0\\
        & & J_t & | &\\
        \hline
        & & & &\\
    \end{array} &\\
    \hline
     & \begin{array}{ccc|cc}
       J_1 & & & &\\
       & \ddots & & &\\
       & & J_t & &\\
       \hline
       - & R_{r_{2}} & - & &\\
       & 0 & & &
    \end{array}
\end{array}
\right).
\end{equation}
Note that $E_1$ and $E_2$ are nilpotent elements of $\mathfrak{gl}_m$ and $\mathfrak{gl}_n$, respectively. By the Jacobson-Morozov theorem, there exist associated $\mathfrak{sl}_2$-triples $\{H_1, E_1, F_1\}$ and $\{H_2, E_2, F_2\}$. Let
\[
H = \left(
    \begin{array}{c|c}
    H_1 & \\
    \hline
     & H_2
    \end{array}
\right), \quad
F = \left(
    \begin{array}{c|c}
    F_1 & \\
    \hline
     & F_2
    \end{array}
\right).
\]
Then, $\{H, E, F\}$ is the associated standard triple containing $E$.

\begin{lemma}\label{lem;lemma2}
    Given a nilpotent element $E_1 \in \mathfrak{sl_m}$ of the form (\ref{eqn;eqn5}), the semisimple element $H_1$ and the nilnegative element $F_1$ described below form the associated $\mathfrak{sl}_2$-triple:
    \begin{align}
    H_1 &= 
    \left(
    \begin{array}{ccc|cccc}
       D_{k_1} & & & & & & \\
        & \ddots & & & & &\\
        & & D_{k_t} & & & &\\
        \hline
        & & & d_1 & & &\\
        & & & & \ddots & &\\
        & & & & & d_{r_1} &\\
        & & & & & & 0\\
    \end{array}
    \right) + \left(
    \begin{array}{ccc|cc}
       I_{k_{i_1}} & & & &\\
        & \ddots & & &\\
        & & I_{k_{i_{r_1}}} & &\\
        \hline
        & & & I_{r_1} &\\
        & & & & 0\\
    \end{array}
    \right), \label{eqn;eqn6}\\
    F_1 &= \left( \begin{array}{ccc|c}
       \tilde{J_1} & & &\quad\\
       & \ddots & &\\
       & & \tilde{J_t} &\\
       \hline
       - & \tilde{C}_{r_{1}} & - &\\
       & 0 & &
    \end{array}
    \right), \nonumber
    \end{align}
    where
    \begin{enumerate}[label = $\bullet$, leftmargin = 0pt]
        \item for each $i = 1, \ldots, t$, recall that $D_{k_i}$ is the diagonal matrix of size $k_i \times k_i$ defined in Proposition \ref{prop;proposition1}.
        \item for each $p = 1, \ldots, r_1$, let 
        \[d_p = - k_{i_p} - 1,\] 
        where $k_{i_p}$ denotes the $k_i$ with the largest index among those appearing in the summation defining $i_p$.
        \item for each $i = 1, \ldots, t$, the block $J_i$ of $E_1$ is replaced by the identity matrix $I_{k_i}$ if $k_1 + \cdots  + k_i \in \{i_1, \ldots , i_{r_1}\}$; in this case, we denote it by $I_{k_{i_p}}$. Otherwise, $J_i$ is replaced by the zero matrix.
        \item \[
    \tilde{J_i} = \begin{pmatrix}
    0 & & & & &\\
    a_1^i & 0 & & & &\\
    & a_2^i & 0 & & &\\
    & & & \ddots & &\\
    & & & & 0 &\\
    & & & & a_{k_i-1}^i & 0\\
\end{pmatrix},
    \]
    where each $a_{\alpha}^i$ is defined inductively as
    \[a_1^i =
    \begin{cases}
        k_i, & \text{if} \ k_1 + \cdots  + k_i \in \{i_1, \ldots , i_{r_1}\}\\
        k_i-1, & \text{otherwise}
    \end{cases}\]
    and
    \[a_{\alpha}^i - a_{\alpha-1}^i =
    \begin{cases}
        (D_{k_i})_{\alpha\alpha} +1, & \text{if} \ k_1 + \cdots  + k_i \in \{i_1, \ldots , i_{r_1}\}\\
        (D_{k_i})_{\alpha\alpha}, & \text{otherwise}
    \end{cases}\]
    for each $\alpha = 2, \ldots, k_i - 1$.
    \item $\tilde{C}_{r_1} = (\tilde{e}_{i_1} \cdots \tilde{e}_{i_{r_1}})^t$, where $\tilde{e}_{i_p} = -(d_p + 1)e_{i_p} = k_{i_p}e_{i_p}$ for each $p = 1, \ldots, r_1$.
    \end{enumerate}
\end{lemma}

\begin{proof}
    Similarly, it is enough to prove that $[H_1, E_1] = 2E_1$, $[H_1, F_1] = -2F_1$ and $[E_1, F_1] = H_1$. Let $H_1^D$ and $H_1^I$ denote the first and second terms of (\ref{eqn;eqn6}), respectively, so that $H_1 = H_1^D + H_1^I$. Using Proposition \ref{prop;proposition1}, we have
    \begin{align*}
    [H_1, E_1]
    &= (H_1^D + H_1^I)E_1 - E_1(H_1^D + H_1^I) \\
    &= H_1^DE_1 - E_1H_1^D \\
    &= 2 \left(\small
    \begin{array}{ccc|c}
        J_1 & & & \quad\\
        & \ddots & &\\
        & & J_t &\\
        \hline
        & & & \\
    \end{array}
    \right) + \left(\small
    \begin{array}{cc|cc}
         & & | &\\
        & & {C_{r_1}}' & 0\\
        & & | & \\
        \hline
        & & &\\
    \end{array}
    \right) - \left(\small
    \begin{array}{cc|cc}
         & & | &\\
        & & {C_{r_1}}'' & 0\\
        & & | & \\
        \hline
        & & & \\
    \end{array}
    \right),
    \end{align*}
    where ${C_{r_1}}' = ((-k_{i_1}+1)e_{i_1} \cdots (-k_{i_{r_1}}+1)e_{i_{r_1}})$ and ${C_{r_1}}'' = ((-k_{i_1}-1)e_{i_1} \cdots (-k_{i_{r_1}}-1)e_{i_{r_1}})$, so that ${C_{r_1}}' - {C_{r_1}}'' = 2{C_{r_1}}$. Hence, $[H_1, E_1] = 2E_1$. Similarly,
    \begin{align*}
    &[H_1, F_1]\\
    &= -2 \left(\small
    \begin{array}{ccc|c}
        J_1 & & & \quad\\
        & \ddots & &\\
        & & J_t &\\
        \hline
        & & & \\
    \end{array}
    \right) + \left(\small
    \begin{array}{ccc|c}
        \; & & & \quad\\
        & \; & &\\
        & & \; &\\
        \hline
        - & {\tilde{C}_{r_1}}' & - & \\
        & 0 & & \\
    \end{array}
    \right) - \left(\small
    \begin{array}{ccc|c}
        \; & & & \quad\\
        & \; & &\\
        & & \; &\\
        \hline
        - & {\tilde{C}_{r_1}}'' & - & \\
        & 0 & & \\
    \end{array}
    \right),
    \end{align*}
    where ${\tilde{C}_{r_1}}' = (d_1k_{i_1}e_{i_1} \cdots d_{r_1}k_{i_{r_1}}e_{i_{r_1}})^t$ and ${\tilde{C}_{r_1}}'' = ((-k_{i_1}+1)k_{i_1}e_{i_1} \cdots (-k_{i_{r_1}}+1)k_{i_{r_1}}e_{i_{r_1}})^t
    $, so that ${\tilde{C}_{r_1}}' - {\tilde{C}_{r_1}}'' = -2{\tilde{C}_{r_1}}$. Hence, $[H_1, F_1] = -2F_1$. Finally,
    \begin{align*}
    [E_1, F_1] = &\left(
    \begin{array}{ccc|c}
        J_1 \tilde{J_1} & & & \quad \\
        & \ddots & &\\
        & & J_t \tilde{J_t} &\\
        \hline
        & & &\\
    \end{array}
    \right) - \left(
    \begin{array}{ccc|cc}
        \tilde{J_1} J_1 & & & \quad & \quad\\
        & \ddots & & & \\
        & & \tilde{J_t} J_t & & \\
        \hline
        & & & \tilde{C}_{r_{1}} C_{r_{1}} &\\
        & & & & 0\\
    \end{array}
    \right)\\
    &+ \left(
    \begin{array}{ccc|c}
        J_{i_1}' & & & \quad \\
        & \ddots & &\\
        & & J_{i_{r_1}}' &\\
        \hline
        & & &\\
    \end{array}
    \right),
    \end{align*}
    where $\tilde{C}_{r_{1}} C_{r_{1}} = (\tilde{e}_{i_1} \cdots \tilde{e}_{i_{r_1}})^t(e_{i_1} \cdots e_{i_{r_1}}) = (k_{i_p})_{p = 1}^{r_1} \in \text{Mat}_{r_1 \times r_1}$. Note that $-k_{i_p} = d_p + 1$. The last summand is defined as $r \times r$ matrix constructed by replacing $J_i$ with $J_{i_p}'$ if $k_1 + \cdots  + k_i = i_p$ and $J_i$ with a zero matrix otherwise. Here, $J_{i_p}'$ is defined as follows for each $p = 1, \ldots, r_1$.
    \[
    J_{i_p}' = \begin{pmatrix}
        0 & & & \\
         & \ddots & & \\
         & & 0 & \\
         & & & k_{i_p}
    \end{pmatrix} \in \text{Mat}_{k_{i_p}\times k_{i_p}}.
    \]
    Moreover, we have
    \begin{equation*}
    J_i \tilde{J_i} - \tilde{J_i} J_i 
    = \begin{pmatrix}
    a_1^i & & & &\\
    & a_2^i - a_1^i & & &\\
    & & \ddots & &\\
    & & & a_{k_i - 1}^i - a_{k_i - 2}^i &\\
    & & & & -a_{k_i - 1}^i\\
\end{pmatrix}.
    \end{equation*}
    If $k_1 + \cdots  + k_i \in \{i_1, \ldots , i_{r_1}\}$, note that
    \begin{align*}
        a_1^i &= k_i = (k_i - 1) + 1 = (D_{k_i})_{11} + 1\\
        a_{\alpha}^i - a_{\alpha-1}^i &= (D_{k_i})_{\alpha\alpha} + 1\\
        -a_{k_i - 1}^i &= -2k_i + 2, \quad \alpha = 2, \ldots, k_i - 1.
    \end{align*}
    Here, if $k_1 + \cdots  + k_i = i_p$, we have $-a_{k_i - 1}^i + k_{i_p} = (-2k_i + 2) + k_i = (-k_i + 1) + 1 = (D_{k_i})_{k_ik_i} + 1$. Therefore,
    \[
    J_i \tilde{J_i} - \tilde{J_i} J_i + J_{i_p}' = D_{k_i} + I_{k_{i_p}}.
    \]
    Otherwise, note that
    \begin{align*}
        a_1^i &= k_i - 1 = (D_{k_i})_{11}\\
        a_{\alpha}^i - a_{\alpha-1}^i &= (D_{k_i})_{\alpha\alpha}\\
        -a_{k_i - 1}^i &= -k_i + 1 = (D_{k_i})_{k_ik_i}, \quad \alpha = 2, \ldots, k_i - 1
    \end{align*}
    Hence, we have $[E_1, F_1] = H_D + H_I = H_1$.
\end{proof}

Similarly, we can work for $E_2$.

\begin{lemma}\label{lem;lemma3}
    Given a nilpotent element $E_2 \in \mathfrak{sl_n}$ of the form (\ref{eqn;eqn5}), the semisimple element $H_2$ and the nilnegative element $F_2$ described below form the associated $\mathfrak{sl}_2$-triple:
    \begin{align*}
    H_2 &= 
    \left(
    \begin{array}{ccc|cccc}
       D_{k_1} & & & & & & \\
        & \ddots & & & & &\\
        & & D_{k_t} & & & &\\
        \hline
        & & & d_1' & & &\\
        & & & & \ddots & &\\
        & & & & & d_{r_2}' &\\
        & & & & & & 0\\
    \end{array}
    \right) - \left(
    \begin{array}{ccc|cc}
       I_{k_{j_1}} & & & &\\
        & \ddots & & &\\
        & & I_{k_{j_{r_2}}} & &\\
        \hline
        & & & I_{r_2} &\\
        & & & & 0\\
    \end{array}
    \right), \\
    F_2 &= \left( \begin{array}{ccc|cc}
       \bar{J_1} & & & | & \\
       & \ddots & & \bar{R}_{r_2} & 0 \\
       & & \bar{J_t} & | & \\
       \hline
       & & & &\\
       & & & &
    \end{array}
    \right), \nonumber
    \end{align*}
    where 
    \begin{enumerate}[label = $\bullet$, leftmargin = 0pt]
        \item for each $i = 1, \ldots, t$, recall that $D_{k_i}$ is the diagonal matrix of size $k_i \times k_i$ defined in Proposition \ref{prop;proposition1}.
        \item for each $q = 1, \ldots, r_2$, let 
        \[d_q' = k_{j_q} + 1,\] 
        when we define $k_{j_q}$ as follows: if $j_q = 1+ k_1 + \cdots + k_j$ for some $j$, then $k_{j_q} = k_{j+1}$; if $j_q = 1$, we set $k_{j_q} = k_1$.
        \item for each $j = 1, \ldots, t$, the block $J_j$ of $E_2$ is replaced by the identity matrix $I_{k_j}$ if $1+ k_1 + \cdots  + k_{j-1} \in \{j_1, \ldots , j_{r_2}\}$; in this case, we denote it by $I_{k_{j_q}}$. Otherwise, $J_j$ is replaced by the zero matrix.
        \item \[
    \bar{J_j} = \begin{pmatrix}
    0 & & & & &\\
    b_1^j & 0 & & & &\\
    & b_2^j & 0 & & &\\
    & & & \ddots & &\\
    & & & & 0 &\\
    & & & & b_{k_j-1}^j & 0\\
\end{pmatrix},
    \]
    where each $b_{\beta}^j$ is defined inductively as
    \[b_{k_j-1}^j =
    \begin{cases}
        k_j, & \text{if} \ 1+ k_1 + \cdots  + k_{j-1} \in \{j_1, \ldots , j_{r_2}\}\\
        k_j-1, & \text{otherwise}
    \end{cases}\]
    and
    \[b_{k_j-\beta}^j - b_{k_j+1-\beta}^j =\small
    \begin{cases}
        (D_{k_j})_{\beta\beta} +1, & \text{if} \ 1+ k_1 + \cdots  + k_{j-1} \in \{j_1, \ldots , j_{r_2}\}\\
        -(D_{k_j})_{(k_j+1-\beta)(k_j+1-\beta)}, & \text{otherwise}
    \end{cases}\]
    for each $\beta = 2, \ldots, k_j - 1$.
    \item $\bar{R}_{r_2} = (\bar{e}_{j_1} \cdots \bar{e}_{j_{r_2}})$, where $\bar{e}_{j_q} = (d_q' - 1)e_{j_q} = k_{j_q}e_{j_ q}$ for each $q = 1, \ldots, r_2$.
    \end{enumerate}
\end{lemma}

Using Lemma \ref{lem;lemma2} and \ref{lem;lemma3}, we can construct the $\mathfrak{sl}_2$-triple $\{H, E, F\}$.
\begin{proposition}\label{prop;proposition2}
    Given an even nilpotent element $E$ of the form (\ref{eqn;eqn5}), the semisimple element $H$ and the nilnegative element $Y$ described below form the associated $\mathfrak{sl}_2$-triple:
    \begin{align}
    H &= \scriptsize \left(
\begin{array}{c|c}
    \begin{array}{ccc|cccc}
        D_{k_1} & & & & & & \\
        & \ddots & & & & & \\
        & & D_{k_t} & & & & \\
        \hline
        & & & d_1 & & & \\
        & & & & \ddots & & \\
        & & & & & d_{r_1} & \\
        & & & & & & 0 \\
    \end{array} &\\
    \hline
     & \begin{array}{ccc|cccc}
        D_{k_1} & & & & & & \\
        & \ddots & & & & & \\
        & & D_{k_t} & & & & \\
        \hline
        & & & d_1' & & & \\
        & & & & \ddots & & \\
        & & & & & d_{r_2}' & \\
        & & & & & & 0 \\
    \end{array}
\end{array}
\right) \label{eqn;eqn8} \\
    &+ \scriptsize \left(
\begin{array}{c|c}
    \begin{array}{ccc|cc}
        I_{k_{i_1}} & & & & \\
        & \ddots & & &  \\
        & & I_{k_{i_{r_1}}} & & \\
        \hline
        & & & I_{r_1} & \\
        & & & & 0 \\
    \end{array} &\\
    \hline
     & \begin{array}{ccc|cc}
        -I_{k_{j_1}} & & & & \\
        & \ddots & & &  \\
        & & -I_{k_{j_{r_2}}} & & \\
        \hline
        & & & -I_{r_2} & \\
        & & & & 0 \\
    \end{array}
\end{array}
\right), \nonumber \\
    F &= \left(
\begin{array}{c|c}
    \begin{array}{ccc|c}
       \tilde{J_1} & & &\quad\\
       & \ddots & &\\
       & & \tilde{J_t} &\\
       \hline
       - & \tilde{C}_{r_{1}} & - &\\
       & 0 & &
    \end{array} &\\
    \hline
     & \begin{array}{ccc|cc}
       \bar{J_1} & & & | & \\
       & \ddots & & \bar{R}_{r_2} & 0 \\
       & & \bar{J_t} & | & \\
       \hline
       & & & &\\
       & & & &
    \end{array}
\end{array}
\right). \nonumber
    \end{align}
    We use the same definitions of $D_{k_i}, d_p, d_q', I_{k_{i_p}}, I_{k_{j_q}}, \tilde{J}_i, \bar{J}_j, \tilde{C}_{r_{1}}, \bar{R}_{r_2}$ as those given in Lemma \ref{lem;lemma2} and \ref{lem;lemma3}.
\end{proposition}

\vspace{1pt}

\section{Jacobson-Morozov Theorem analogue in $\mathfrak{gl}(m|n)$}

In this section, we state and prove an analogue of the Jacobson-Morozov theorem for the Lie superalgebra $\mathfrak{gl}(m|n)$, which constitutes our main result. In contrast to the case of $\mathfrak{sl}_m$, not every odd nilpotent in $\mathfrak{gl}(m|n)$ admits an $\mathfrak{osp}(1|2)$-subalgebra.

Let $V = V_{\bar{0}} \oplus V_{\bar{1}}$ be a vector superspace with $\dim V = (m|n)$. We replace the basis
\[ \mathcal{B} = \{v_{\bar{1}}, \ldots, v_{\bar{m}}; v_1, \ldots, v_n\}\]
of $V$ with a disjoint union of bases whose elements alternate in parity. For example, when $m \leq n$, we may take
\[\mathcal{B}' = \{v_{\bar{1}}, v_1, v_{\bar{2}}, v_2, \ldots, v_{\bar{m}}\} \sqcup \{v_m\} \sqcup \cdots \sqcup \{v_n\}.\]
\begin{definition}
    Let $x \in \text{End}V$. Then a \textbf{super Jordan block} of $x$ of size $r$ is the matrix associated with an ordered basis $\{v_r, v_{r-1}, \ldots , v_1\}$ consisting of alternating parities and satisfying
    \[xv_i = \lambda v_i + v_{i+1}, \quad 1 \leq i \leq r-1  \quad \text{and} \quad  xv_r = \lambda v_r.\]
    In matrix form, we have
    \[
    J_{\lambda, r}^{\text{super}} = \begin{pmatrix}
        \lambda & 1 & 0 & \cdots & 0 \\
        0 & \lambda & 1 & \cdots & 0 \\
        \vdots & \vdots & \ddots & \ddots & \vdots \\
        0 & 0 & \cdots & \lambda & 1 \\
        0 & 0 & \cdots & 0 & \lambda
    \end{pmatrix}.
    \]
    We define a \textbf{super Jordan matrix} as a block diagonal matrix composed of super Jordan blocks.
\end{definition}

\begin{example}
    In $\mathfrak{gl}(4|2)$, consider
    \[
    x= \left(
    \begin{array}{cccc|cc}
        & & & & 1 & \\
       & & & & & 1 \\
       & & & & & \\
       & & & & & \\
       \hline
       & & 1 & & &\\
       & & & 1 & &
    \end{array}
    \right) \in \mathfrak{gl}(4|2).
    \]
    The corresponding super Jordan matrix of $x$ is the matrix associated with the ordered basis 
    \[\{e_{\bar{2}}, e_2, e_{\bar{4}}\} \;\sqcup\; \{e_{\bar{1}}, e_1, e_{\bar{3}}\},\] 
    where each $e_i$ denotes a standard unit vector. Therefore, the super Jordan matrix of $x$ is given by
    \[
    \left(
    \begin{array}{c|c}
       J_{0, 3}^{\text{super}} & \\
       \hline
       & J_{0, 3}^{\text{super}}
    \end{array}
    \right)
    =
    \left(
    \begin{array}{ccc|ccc}
       0 & 1 & & & & \\
       & 0 & 1 & & & \\
       & & 0 & & & \\
       \hline
       & & & 0 & 1 & \\
       & & & & 0 & 1 \\
       & & & & & 0
    \end{array}
    \right).
    \]
\end{example}

\vspace{1pt}

This theorem constitutes the main result of this paper.
\begin{theorem}\label{thm;theorem1}
    Let $\mathfrak{g} = \mathfrak{gl}(m|n)$, and $\text{Lie}(G_{\bar{0}}) = \mathfrak{g}_{\bar{0}}$ for some semisimple, simply connected group $G_{\bar{0}}$. Then an odd nilpotent element $e \in \mathfrak{g}$ is contained in an $\mathfrak{osp}(1|2)$-subalgebra of $\mathfrak{g}$ if and only if $e$ lies in the $G_{\bar{0}}$-orbit of a super Jordan matrix consisting entirely of super Jordan blocks of odd size.
\end{theorem}


\vspace{1pt}

\subsection{Construction of $\mathfrak{osp}(1|2)$-subalgebra}

Let $f := [F, e]$ be another odd element of $\mathfrak{gl}(m|n)$. We now simplify the conditions under which the subalgebra $\langle e, f, H, E, F\rangle$ is isomorphic to $\mathfrak{osp}(1|2)$ by performing a series of explicit calculations. Note that $\mathfrak{osp}(1|2)$ is generated by even elements $\{H, E, F\}$ and odd elements $\{e, f\}$, satisfying the following defining relations:
\begin{align*}
    &[H, E] = 2E, \ [H, F] = -2F, \ [E, F] = H,\\
    &[H, e] = e, \ [H, f] = -f,\\
    &[E, f] = e, \ [E, e] = 0,\\
    &[F, e] = f, \ [F, f] = 0,\\
    &[e, e] = 2E, \ [f, f] = -2F, \ [e, f] = -H
\end{align*}

\begin{lemma}\label{lem;lemma4}
    Given even generators $\{H, E, F\}$ that forms $\mathfrak{sl}_2$-subalgebra and odd generators $\{e, f\}$ with $E = e^2$ and $f = [F,e]$,
    \[
    \langle e, f, H, E, F\rangle \simeq \mathfrak{osp}(1|2) \quad \text{if and only if} \quad [H,e] = e \quad \text{and} \quad [F,f] = 0.
    \]
\end{lemma}
\begin{proof}
    It suffices to prove the ``if'' part. We verify the relations listed above using the Jacobi identity. Indeed,
    \begin{align*}
        [H,f] &= [[H,F],e] + [F,[H,e]] = -2[F,e] + [F,e] = -f, \\
        [E,e] &= Ee - eE = 0, \\
        [E,f] &= [[E,F],e] + [F,[E,e]] = [H,e] = e, \\
        [e,e] &= 2 e^2 = 2E, \\
        [e,f] &= [[e,F],e] + [F,[e,e]] = -[e,f] - 2[E,F] \quad \text{implies} \quad [e,f] = -H, \\
        [f,f] &= [[f,F],e] + [F,[f,e]] = -[F,H] = -2F.
    \end{align*}
\end{proof}

    Given an odd nilpotent element $e \in \mathfrak{gl}(m|n)$ of the form (\ref{eqn;eqn4}), we denote certain blocks $J_{i_p}$ and $J_{j_q}$ by $J_i \in \{J_1, \ldots, J_t\}$\footnote{We distinguish $J_i$ from $J_{i'}$ whenever $k_1 + \cdots +k_i \neq k_1+\cdots + k_{i'}$, even though $k_i = k_{i'}$.} when $k_1 + \cdots + k_i = i_p$ and $1+k_1 + \cdots + k_{i-1} = j_q$, respectively.
    
\begin{lemma}\label{lem;lemma5}
    Let $e \in \mathfrak{gl}(m|n)$ be an odd nilpotent element of the form (\ref{eqn;eqn4}), with $E = e^2$, and let $\{H, E, F\}$ be an associated $\mathfrak{sl}_2$-triple for $E$. Define $f := [F,e]$. Then,
    \begin{align*}
\langle e, f, H, E, F \rangle &\simeq \mathfrak{osp}(1|2) \quad \text{if and only if} \\
\{J_1, \ldots, J_t\} &= \{J_{i_1}, \ldots, J_{i_{r_1}}\} 
\; \sqcup \; \{J_{j_1}, \ldots, J_{j_{r_2}}\} 
\quad \text{and} \quad s = 0.
\end{align*}
\end{lemma}

To prove this, we introduce the following lemmas.
\begin{lemma}\label{lem;lemma6}
    Under the same assumptions as in Lemma \ref{lem;lemma5},
    \begin{align*}
    [H,e] &= e \quad \text{if and only if}\\
    \{J_1, \ldots, J_t\} &= \{J_{i_1}, \ldots, J_{i_{r_1}}\} \; \sqcup \; \{J_{j_1}, \ldots, J_{j_{r_2}}\} 
    \quad \text{and} \quad s = 0.
    \end{align*}
\end{lemma}

\begin{lemma}\label{lem;lemma7}
    Under the same assumptions as in Lemma \ref{lem;lemma5},
    \[
    [H,e] = e \quad \text{implies} \quad [F,f] = 0.
    \]
\end{lemma}

\begin{proof}[Proof of the Lemma \ref{lem;lemma6}]
    Let $H_D$ and $H_I$ denote the first and second terms of (\ref{eqn;eqn8}), respectively, so that $H = H_D + H_I$. Then we have
    \[
        [H,e] = (H_D + H_I)e - e(H_D + H_I) = (H_De-eH_D) + (H_Ie-eH_I).
    \]
    Let $A = H_De-eH_D$ and $B = H_Ie-eH_I$. Then we get
    \begin{align*}
        A &= \left(
\begin{array}{c|c}
     & \begin{array}{ccc|cc}
        D_{k_1} & & & \quad & \quad \\
        & \ddots & & & \\
        & & D_{k_t} & & \\
        \hline
        & & & & \\
        & & & & \\
    \end{array}\\
    \hline
    \begin{array}{ccc|cc}
        D_{k_1}J_1 & & & | & \\
        & \ddots & & \hat{C}_{r_1} & 0 \\
        & & D_{k_t}J_t & | & \\
        \hline
        - & \hat{R}_{r_2} & - & & \\
        & 0 & & & \\
    \end{array} & 
\end{array}
\right)\\
    &- \left(
\begin{array}{c|c}
     & \begin{array}{ccc|cc}
        D_{k_1} & & & \quad & \quad \\
        & \ddots & & & \\
        & & D_{k_t} & & \\
        \hline
        & & & & \\
        & & & & \\
    \end{array}\\
    \hline
    \begin{array}{ccc|cc}
        J_1D_{k_1} & & & | & \\
        & \ddots & & \hat{\hat{C}}_{r_1} & 0 \\
        & & J_tD_{k_t} & | & \\
        \hline
        - & \hat{\hat{R}}_{r_2} & - & & \\
        & 0 & & & \\
    \end{array} & 
\end{array}
\right),
    \end{align*}
    where for each $p = 1, \ldots, r_1$ and $q = 1, \ldots, r_2$,
    \begin{align*}
        \hat{C}_{r_1} &= \{\hat{e}_{i_1}, \ldots, \hat{e}_{i_{r_1}}\}, \quad \text{each} \quad \hat{e}_{i_p} = (-k_{i_p}+1)e_{i_p},\\
        \hat{R}_{r_2} &= \{\hat{e}_{j_1}, \ldots, \hat{e}_{j_{r_2}}\}, \quad \text{each} \quad \hat{e}_{j_q} = (k_{j_q}+1)e_{j_q},\\
        \hat{\hat{C}}_{r_1} &= \{\hat{\hat{e}}_{i_1}, \ldots, \hat{\hat{e}}_{i_{r_1}}\}, \quad \text{each} \quad \hat{\hat{e}}_{i_p} = (-k_{i_p}-1)e_{i_p},\\
        \hat{\hat{R}}_{r_2} &= \{\hat{\hat{e}}_{j_1}, \ldots, \hat{\hat{e}}_{j_{r_2}}\}, \quad \text{each} \quad \hat{\hat{e}}_{j_q} = (k_{j_q}-1)e_{j_q}.
    \end{align*}
    Therefore, $\hat{C}_{r_1} - \hat{\hat{C}}_{r_1} = 2C_{r_1}$ and $\hat{R}_{r_2} - \hat{\hat{R}}_{r_2} = 2R_{r_2}$.\\
    In addition, we have $D_{k_i}J_i - J_iD_{k_i} = 2J_i$ for each $i = 1, \ldots, t.$\\
    Hence,
    \[
    A = 2\left(
\begin{array}{c|c}
     & \begin{array}{ccccc}
        \quad & \quad & \quad & \quad & \quad \\
        & & & & \\
        & & & & \\
        & & & & \\
        & & & & \\
    \end{array}\\
    \hline
    \begin{array}{ccc|cc}
        J_1 & & & | & \\
        & \ddots & & C_{r_1} & 0 \\
        & & J_t & | & \\
        \hline
        - & R_{r_2} & - & & \\
        & 0 & & & \\
    \end{array} & 
\end{array}
\right)
    \]
    Moreover,
    \begin{align*}
        B &= \left(
\begin{array}{c|c}
     & \begin{array}{ccc|cc}
        I_{k_{i_1}} & & & \quad & \quad \\
        & \ddots & & & \\
        & & I_{k_{i_{r_1}}} & & \\
        \hline
        & & & & \\
        & & & & \\
    \end{array}\\
    \hline
    \begin{array}{ccc|cc}
        -J_{j_1} & & & | & \\
        & \ddots & & -C_{r_1}^* & 0 \\
        & & -J_{j_{r_2}} & | & \\
        \hline
        - & -R_{r_2} & - & & \\
        & 0 & & & \\
    \end{array} & 
\end{array}
\right)\\
    &+ \left(
\begin{array}{c|c}
     & \begin{array}{ccc|cc}
        I_{k_{j_1}} & & & \quad & \quad \\
        & \ddots & & & \\
        & & I_{k_{j_{r_2}}} & & \\
        \hline
        & & & & \\
        & & & & \\
    \end{array}\\
    \hline
    \begin{array}{ccc|cc}
        -J_{i_1} & & & | & \\
        & \ddots & & -C_{r_1} & 0 \\
        & & -J_{i_{r_1}} & | & \\
        \hline
        - & -R_{r_2}^* & - & & \\
        & 0 & & & \\
    \end{array} & 
\end{array}
\right),
    \end{align*}
    where
    \[
        C_{r_1}^* = \begin{pmatrix}
            I_{k_{j_1}} & & \\
        & \ddots & \\
        & & I_{k_{j_{r_2}}}
        \end{pmatrix}C_{r_1}, \quad R_{r_2}^* = R_{r_2}\begin{pmatrix}
            I_{k_{i_1}} & & \\
        & \ddots & \\
        & & I_{k_{i_{r_1}}}
        \end{pmatrix}
    \]
    Therefore,
    \begin{align*}
        [H,e] = A+B = e \quad &\text{if and only if}\\
        \small\begin{pmatrix}
            I_{k_{i_1}} & & \\
        & \ddots & \\
        & & I_{k_{i_{r_1}}}
        \end{pmatrix} &+ \small\begin{pmatrix}
            I_{k_{j_1}} & & \\
        & \ddots & \\
        & & I_{k_{j_{r_2}}}
        \end{pmatrix} = I_r,
        \; C_{r_1}^* = R_{r_2}^* = 0 \; \text{and} \; s = 0. \\
        &\text{if and only if}\\
        \small
        \begin{pmatrix}
            I_{k_{i_1}} & & \\
        & \ddots & \\
        & & I_{k_{i_{r_1}}}
        \end{pmatrix} &+ \small\begin{pmatrix}
            I_{k_{j_1}} & & \\
        & \ddots & \\
        & & I_{k_{j_{r_2}}}
        \end{pmatrix} = I_r \; \text{and} \; s = 0.\\
        &\text{if and only if}\\
        \{J_1, \ldots, J_t\} &= \{J_{i_1}, \ldots, J_{i_{r_1}}\} \;\sqcup\; \{J_{j_1}, \ldots, J_{j_{r_2}}\} \; \text{and} \; s = 0.
    \end{align*}
    \end{proof}

    \begin{proof}[Proof of the Lemma \ref{lem;lemma7}]

    Let us compute $f$.
    \[
    f = [F, e] = Fe - eF.
    \]
    Here, we have
    \begin{align}
        Fe &= \left(
\begin{array}{c|c}
     & \begin{array}{ccc|c}
        \tilde{J_1} & & & \quad \\
        & \ddots & & \\
        & & \tilde{J_t} & \\
        \hline
        - & \tilde{C}_{r_1} & - & \\
        & 0 & & \\
    \end{array}\\
    \hline
    \begin{array}{ccc|c}
        \bar{J_1}J_1 & & & \quad \\
        & \ddots & & \\
        & & \bar{J_t}J_t & \\
        \hline
        & & & \\
    \end{array} & 
\end{array}
\right) \label{eqn;eqn9}\\
    &+ \left(
\begin{array}{c|c}
     & \begin{array}{cccc}
        \quad & \quad & \quad & \quad \\
        & & & \\
        & & & \\
        & & & \\
    \end{array}\\
    \hline
    \begin{array}{ccc|c}
        J_{j_1}' & & & \quad \\
        & \ddots & & \\
        & & J_{j_{r_2}}' & \\
        \hline
        & & & \\
    \end{array} & 
\end{array}
\right), \nonumber
    \end{align}
    where for each $q = 1, \ldots, r_2$,
    \[
    J_{j_{q}}' = \begin{pmatrix}
        k_{j_q} & & & \\
        & 0 & & \\
        & & \ddots & \\
        & & & 0
    \end{pmatrix}.
    \]
    Moreover, we have
    \begin{align}
        eF &= \left(
\begin{array}{c|c}
     & \begin{array}{ccc|cc}
        \bar{J_1} & & & | & \\
        & \ddots & & \bar{R}_{r_2} & 0 \\
        & & \bar{J_t} & | & \\
        \hline
        & & & & \\
    \end{array}\\
    \hline
    \begin{array}{ccc|c}
        J_1\tilde{J_1} & & & \quad \\
        & \ddots & & \\
        & & J_t\tilde{J_t} & \\
        \hline
        & & & \\
    \end{array} & 
\end{array}
\right) \label{eqn;eqn10}\\
    &+ \left(
\begin{array}{c|c}
     & \begin{array}{cccc}
        \quad & \quad & \quad & \quad \\
        & & & \\
        & & & \\
        & & & \\
    \end{array}\\
    \hline
    \begin{array}{ccc|c}
        J_{i_1}'' & & & \quad \\
        & \ddots & & \\
        & & J_{i_{r_1}}'' & \\
        \hline
        & & & \\
    \end{array} & 
\end{array}
\right), \nonumber
    \end{align}
    where for each $p = 1, \ldots, r_1$,
    \[
    J_{i_{p}}'' = \begin{pmatrix}
        0 & & & \\
        & \ddots & & \\
        & & 0 & \\
        & & & k_{i_p}
    \end{pmatrix}.
    \]
    Hence, we have
    \begin{align}
    f &= \scriptsize \left(
\begin{array}{c|c}
     & \begin{array}{ccc|cc}
        \tilde{J_1} - \bar{J_1} & & & | & \\
        & \ddots & & -\bar{R}_{r_2} & 0 \\
        & & \tilde{J_t} - \bar{J_t} & | & \\
        \hline
        - & \tilde{C}_{r_1} & - & & \\
        & 0 & & &
    \end{array}\\
    \hline
    \begin{array}{ccc|c}
        \bar{J_1}J_1 - J_1\tilde{J_1} & & & \quad \\
        & \ddots & & \\
        & & \bar{J_t}J_t - J_t\tilde{J_t} & \\
        \hline
        & & & \\
    \end{array} & 
\end{array}
\right) \label{eqn;eqn11}\\
    & + I_1 - I_2, \nonumber
    \end{align}
    where $I_1$ and $I_2$ denote the first terms of (\ref{eqn;eqn9}) and (\ref{eqn;eqn10}), respectively. Let $f^*$ denote the first term of (\ref{eqn;eqn11}). Then,
    \begin{align*}
        [F,f] &= F(f^*+I_1-I_2) - (f^*+I_1-I_2)F \\
        &= Ff^* - f^*F + (FI_1 - I_1F) - (FI_2 - I_2F).
    \end{align*}
    Here,
    \begin{align*}
        Ff^* = \left( 
        \begin{array}{c|c}
        & \scalebox{0.65}{$\begin{array}{ccc|cc}
        \tilde{J_1}(\tilde{J_1} - \bar{J_1}) & & & | & \\
        & \ddots & & -\bar{R}_{r_2}' & 0 \\
        & & \tilde{J_t}(\tilde{J_t} - \bar{J_t}) & | & \\
        \hline
        - & \tilde{C}_{r_1}' & - & -M & \\
        & 0 & & & 0
    \end{array}$}\\
    \hline
     \scalebox{0.65}{$\begin{array}{ccc|c}
        \bar{J_1}(\bar{J_1}J_1 - J_1\tilde{J_1}) & & & \quad \\
        & \ddots & & \\
        & & \bar{J_t}(\bar{J_t}J_t - J_t\tilde{J_t}) & \\
        \hline
        & & & \\
    \end{array}$} & 
\end{array}
\right),
    \end{align*}
    where 
    {\scriptsize
    \begin{align*}
        \bar{R}_{r_2}' &= (\bar{e}_{j_1 + 1}' \cdots \bar{e}_{j_{r_2} + 1}' ), \quad \bar{e}_{j_q + 1}' = \begin{cases}
            k_{j_q}^2e_{j_q + 1} & \text{if } k_1 + \cdots + k_{j_q} \in \{i_1, \ldots, i_{r_1}\}, \\
            k_{j_q}(k_{j_q} - 1)e_{j_q + 1} & \text{otherwise},
        \end{cases}\\
        \tilde{C}_{r_1}' &= (\tilde{e}_{i_1 - 1}' \cdots \tilde{e}_{i_{r_1} + 1}' ), \quad \tilde{e}_{i_p - 1}' = \begin{cases}
            k_{i_p}(k_{i_p} - 2)e_{i_p - 1} & \text{if } 1+ k_1 + \cdots + k_{i_p - 1} \in \{j_1, \ldots, j_{r_2}\}, \\
            k_{i_p}(k_{i_p} - 1)e_{i_p - 1} & \text{otherwise},
        \end{cases}
    \end{align*}
    }
    and $M = \tilde{C}_{r_1}\bar{R}_{r_2} \in \text{Mat}_{r_1 \times r_2}$.
    Moreover, we have
    \begin{align*}
        f^*F = \left( 
        \begin{array}{c|c}
        & \scalebox{0.65}{$\begin{array}{ccc|cc}
        (\tilde{J_1} - \bar{J_1})\bar{J_1} & & & | & \\
        & \ddots & & \bar{R}_{r_2}'' & 0 \\
        & & (\tilde{J_t} - \bar{J_t})\bar{J_t} & | & \\
        \hline
        - & \tilde{C}_{r_1}'' & - & M & \\
        & 0 & & & 0
    \end{array}$}\\
    \hline
     \scalebox{0.65}{$\begin{array}{ccc|c}
        (\bar{J_1}J_1 - J_1\tilde{J_1})\tilde{J_1} & & & \quad \\
        & \ddots & & \\
        & & (\bar{J_t}J_t - J_t\tilde{J_t})\tilde{J_t} & \\
        \hline
        & & & \\
    \end{array}$} & 
\end{array}
\right),
    \end{align*}
    where 
    {\scriptsize
    \begin{align*}
        \bar{R}_{r_2}'' &= (\bar{e}_{j_1 + 1}'' \cdots \bar{e}_{j_{r_2} + 1}'' ), \quad \bar{e}_{j_q + 1}'' = \begin{cases}
            k_{j_q}(2 - k_{j_q})e_{j_q + 1} & \text{if } k_1 + \cdots + k_{j_q} \in \{i_1, \ldots, i_{r_1}\}, \\
            k_{j_q}(1 - k_{j_q})e_{j_q + 1} & \text{otherwise},
        \end{cases}\\
        \tilde{C}_{r_1}'' &= (\tilde{e}_{i_1 - 1}'' \cdots \tilde{e}_{i_{r_1} + 1}'' ), \quad \tilde{e}_{i_p - 1}'' = \begin{cases}
            k_{i_p}^2e_{i_p - 1} & \text{if } 1+ k_1 + \cdots + k_{i_p - 1} \in \{j_1, \ldots, j_{r_2}\}, \\
            k_{i_p}(k_{i_p} - 1)e_{i_p - 1} & \text{otherwise}.
        \end{cases}
    \end{align*}
    }
    In addition, let $F_1 = FI_1 - I_1F$ and $F_2 = FI_2 - I_2F$. Then we have
    {\scriptsize
    \begin{align*}
        F_1 &= \left(
\begin{array}{c|c}
     & \begin{array}{cccc}
        \quad & \quad & \quad & \quad \\
        & & & \\
        & & & \\
        & & & \\
    \end{array}\\
    \hline
    \begin{array}{ccc|c}
        \bar{J}_{j_1}J_{j_1}' - J_{j_1}'\bar{J}_{j_1} & & & \quad \\
        & \ddots & & \\
        & & \bar{J}_{j_{r_2}}J_{j_{r_2}}' - J_{j_{r_2}}'\bar{J}_{j_{r_2}} & \\
        \hline
        & & & \\
    \end{array} & 
\end{array}
\right), \\
        F_2 &= \left(
\begin{array}{c|c}
     & \begin{array}{cccc}
        \quad & \quad & \quad & \quad \\
        & & & \\
        & & & \\
        & & & \\
    \end{array}\\
    \hline
    \begin{array}{ccc|c}
        \bar{J}_{i_1}J_{i_1}'' - J_{i_1}''\bar{J}_{i_1} & & & \quad \\
        & \ddots & & \\
        & & \bar{J}_{i_{r_1}}J_{i_{r_1}}'' - J_{i_{r_1}}''\bar{J}_{i_{r_1}} & \\
        \hline
        & & & \\
    \end{array} & 
\end{array}
\right).
    \end{align*}
    }
    To find the necessary and sufficient condition for $[F,f] = 0$, we first compute
    \[\tilde{J}_i(\tilde{J}_i - \bar{J}_i) - (\tilde{J}_i - \bar{J}_i)\bar{J}_i = \tilde{J}_i^2 - 2\tilde{J}_i\bar{J}_i + \bar{J}_i^2\]
    in each case:
    \begin{enumerate}[label=(\Roman*)]
        \item \label{case;case1} $k_1 + \cdots + k_i \in \{i_1, \ldots, i_{r_1} \}$ and $1 + k_1 + \cdots + k_{i-1} \in \{ j_1, \ldots, j_{r_2} \}$
        \item \label{case;case2} $k_1 + \cdots + k_i \in \{i_1, \ldots, i_{r_1} \}$ and $1 + k_1 + \cdots + k_{i-1} \notin \{ j_1, \ldots, j_{r_2} \}$
        \item \label{case;case3} $k_1 + \cdots + k_i \notin \{i_1, \ldots, i_{r_1} \}$ and $1 + k_1 + \cdots + k_{i-1} \in \{ j_1, \ldots, j_{r_2} \}$
        \item \label{case;case4} otherwise
    \end{enumerate}
    We use the notation $\tilde{J}_i$ and $\bar{J}_i$ from Lemmas \ref{lem;lemma2} and \ref{lem;lemma3}. In the case of \ref{case;case1}, we have
    \[
    \tilde{J}_i^2 - 2\tilde{J}_i\bar{J}_i + \bar{J}_i^2 \neq 0.
    \]
    However, in the cases of \ref{case;case2}, \ref{case;case3}, and \ref{case;case4}, we have
    \[
    \tilde{J}_i^2 - 2\tilde{J}_i\bar{J}_i + \bar{J}_i^2 = 0.
    \]
    Therefore, by considering the forms of $\bar{R}_{r_2}'$, $\tilde{C}_{r_1}'$, and $M$, the following condition must be satisfied in order for the upper-right block of $[F,f]$ to vanish:
    \begin{equation}\label{eqn;eqn12}
    \{J_{i_1}, \ldots, J_{i_{r_1}}\} \cap \{J_{j_1}, \ldots, J_{j_{r_2}}\} = \emptyset.
    \end{equation}

    \vspace{1pt}
    
    We now compute the lower-left block of $[F,f]$. Initially, we compute 
    \[\bar{J}_i(\bar{J}_iJ_i - J_i\tilde{J}_i) - (\bar{J}_iJ_i - J_i\tilde{J}_i)\tilde{J}_i = \bar{J}_i^2J_i - 2\bar{J}_iJ_i\tilde{J}_i + J_i\tilde{J}_i^2.\]
    As before, we consider the cases \ref{case;case1}–\ref{case;case4}. We also use the notation $\tilde{J}_i$ and $\bar{J}_i$ from Lemmas \ref{lem;lemma2} and \ref{lem;lemma3}. In the case of \ref{case;case1}, we have
    \[
    \bar{J}_i^2J_i - 2\bar{J}_iJ_i\tilde{J}_i + J_i\tilde{J}_i^2 = \begin{pmatrix}
        0 & & & & & \\
        -2k_i(k_i-1) & 0 & & & & \\
        & 0 & 0 & & & \\
        & & \ddots & \ddots & & \\
        & & & 0 & 0 & \\
        & & & & -2k_i(k_i-1) & 0
    \end{pmatrix}.
    \]
    In the case of \ref{case;case2}, we obtain
    \[
    \bar{J}_i^2J_i - 2\bar{J}_iJ_i\tilde{J}_i + J_i\tilde{J}_i^2 = \begin{pmatrix}
        0 & & & & & \\
        0 & 0 & & & & \\
        & 0 & 0 & & & \\
        & & \ddots & \ddots & & \\
        & & & 0 & 0 & \\
        & & & & -2k_i(k_i-1) & 0
    \end{pmatrix}.
    \]
    In the case of \ref{case;case3}, we obtain
    \[
    \bar{J}_i^2J_i - 2\bar{J}_iJ_i\tilde{J}_i + J_i\tilde{J}_i^2 = \begin{pmatrix}
        0 & & & & & \\
        -2k_i(k_i-1) & 0 & & & & \\
        & 0 & 0 & & & \\
        & & \ddots & \ddots & & \\
        & & & 0 & 0 & \\
        & & & & 0 & 0
    \end{pmatrix}.
    \]
    In the case of \ref{case;case4}, we obtain
    \[ \scriptsize
    \bar{J}_i^2J_i - 2\bar{J}_iJ_i\tilde{J}_i + J_i\tilde{J}_i^2 = \begin{pmatrix}
        0 & & & & & \\
        -2(k_i-1) & 0 & & & & \\
        & \ddots & & & & \\
        & & -2\alpha(k_i - \alpha) & \ddots & & \\
        & & & \ddots & 0 & \\
        & & & & -2(k_i-1) & 0
    \end{pmatrix}.
    \]
    The remaining part of the lower-left block of $[F,f]$ is given by $F_1$ and $F_2$. Note that
    \begin{align*}
        \bar{J}_{j_q}J_{j_q}' - J_{j_q}'\bar{J}_{j_q} &= \begin{pmatrix}
        0 & & & & \\
        2k_{j_q}(k_{j_q}-1) & 0 & & & \\
        & & \ddots & & \\
        & & & & 0 \\
    \end{pmatrix},\\
        \tilde{J}_{i_p}J_{i_p}'' - J_{i_p}''\tilde{J}_{i_p} &=
        \begin{pmatrix}
        0 & & & & \\
        & & \ddots & & \\
        & & & 0 & \\
        & & & -2k_{i_p}(k_{i_p}-1) & 0 \\
    \end{pmatrix}
    \end{align*}
    for each $q = 1, \ldots, r_2$ and $p = 1, \ldots, r_1$.

    In order for $[F,f] = 0$, the condition (\ref{eqn;eqn12}) must be satisfied. Hence, the case \ref{case;case1} of the lower-left block can be disregarded. Moreover, to ensure that the lower-left block of $[F,f]$ vanishes, the case \ref{case;case4} must be discarded unless $k_i = 1$, since its nonzero entries cannot be canceled by $F_1$ and $F_2$. Therefore,
    \begin{align*}
[F,f] = 0 \quad &\text{if and only if} \\[2pt]
\{J_{i_1}, \ldots, J_{i_{r_1}}\} &\cap \{J_{j_1}, \ldots, J_{j_{r_2}}\} = \emptyset, \quad \text{and for each } i, \\
&\text{either case } \ref{case;case2}, \text{ case } \ref{case;case3}, \text{ or } k_i = 1.
\end{align*}
    Hence, if $[H,e]=e$, we have 
    \[
    \{J_1, \ldots, J_t\} = \{J_{i_1}, \ldots, J_{i_{r_1}}\} \sqcup \{J_{j_1}, \ldots, J_{j_{r_2}}\}
    \]
    and this decomposition satisfies the conditions on the right-hand side.
\end{proof}

\begin{proof}[Proof of the Lemma \ref{lem;lemma5}]
    By Lemma \ref{lem;lemma4}, it suffices to prove the following.
    \begin{align*}
[H,e] = e \quad &\text{and} \quad [F,f] = 0 
\quad \text{if and only if} \\[2pt]
\{J_1, \ldots, J_t\} &= \{J_{i_1}, \ldots, J_{i_{r_1}}\} \;\sqcup\; \{J_{j_1}, \ldots, J_{j_{r_2}}\}, 
\quad \text{and} \quad s = 0.
\end{align*}
    The forward implication is proved by Lemma \ref{lem;lemma6}.\\
    The converse is proved by Lemmas \ref{lem;lemma6} and \ref{lem;lemma7}.
\end{proof}

\vspace{1pt}

\subsection{Proof of the theorem}
It now suffices to prove the following.
\begin{lemma}\label{lem;lemma8}
    Let $e \in \mathfrak{g} = \mathfrak{gl}(m|n)$ be an odd nilpotent element. Then $e$ lies in the $G_{\bar{0}}$-orbit of a super Jordan matrix composed entirely of super Jordan blocks of odd size if and only if
\[
\{J_1, \ldots, J_t\} = \{J_{i_1}, \ldots, J_{i_{r_1}}\} 
\;\sqcup\; \{J_{j_1}, \ldots, J_{j_{r_2}}\}, 
\quad \text{and} \quad s = 0.
\]
\end{lemma}

\begin{proof} First, assume the left-hand side. We denote the odd nilpotent element by $x$ instead of $e$ throughout this proof, in order to avoid confusion with the standard unit vectors $e_i$. In addition, for convenience, we parametrize an ordered basis by the set 
\[\{1, \ \ldots \ , \ m \ ; \ m+1, \ \ldots \ , m+n\}\] 
instead of 
\[I(m|n) = \{\bar{1}, \ \ldots \ , \ \bar{m} \ ; \ 1, \ \ldots \ , n\}.\] Recall that the odd nilpotent element $x \in \mathfrak{gl}(m|n))$ has the form in (\ref{eqn;eqn4}). Given $C_{r_1} = (e_{i_1}\cdots e_{i_{r_1}})$, take a standard unit vector $e_{r+p}$ with $p = 1, \ldots, r_1$. Then there is no standard unit vector $e_i$ such that $xe_i = e_{r+p}$. Consequently, starting from $e_{r+p}$ and applying $x$ repeatedly yields the sequence of standard unit vectors
    \[
    e_{r+p},\ xe_{r+p},\ x^2e_{r+p},\ \ldots \ , \ x^{2k_{i_p}}e_{r+p},
    \]
    where each term satisfies
    \[x^{2u}e_{r+p} = e_{i_p - u +1}, \quad  \text{and} \quad x^{2u-1}e_{r+p} = e_{m+i_p-u+1}\]
    for $u = 1, \ldots, k_{i_p}$. Moreover, $x^{2k_{i_p}+1}e_{r+p} = 0$. Therefore, the ordered set 
    \[\{x^{2k_{i_p}}e_{r+p}, \ x^{2k_{i_p}-1}e_{r+p}, \ \ldots \ , \ xe_{r+p},\ e_{r+p} \}\]
    forms a super Jordan block of size $2k_{i_p}+1$, which is odd.
    Next, given $R_{r_2} = (e_{j_1} \ldots e_{j_{r_2}})^t$, let $\bar{q} = j_q - 1+k_{j_q}$ and take a standard unit vector $e_{m+\bar{q}}$, where $q = 1, \ldots, r_2$. Note that \[\bar{q} = j_q - 1+k_{j_q} = k_1+ \cdots +k_{j_q}.\]
    Then there is no standard unit vector $e_i$ such that $xe_i = e_{m+\bar{q}}$. Hence, we obtain the following sequence of standard unit vectors
    \[
    e_{m+\bar{q}}, \ xe_{m+\bar{q}}, \ x^2e_{m+\bar{q}}, \ \ldots \ ,  \ x^{2k_{j_q}}e_{m+\bar{q}},
    \]
    where each term satisfies 
    \[x^{2u}e_{m+\bar{q}} = e_{m+\bar{q} - u}, \quad \text{for} \; u = 1, \ldots k_{j_q}-1,\] 
    and 
    \[x^{2u-1}e_{m+\bar{q}} = e_{\bar{q} - u+1}, \quad \text{for} \; u = 1, \ldots k_{j_q}.\]
    Note that 
    \[x^{2k_{j_q}-1}e_{m+\bar{q}} = e_{\bar{q} - k_{j_q} +1} = e_{jq}, \quad \text{and} \quad x^{2k_{j_q}}e_{m+\bar{q}} = xe_{jq} = e_{m+r+q}.\]
    Moreover, $x^{2k_{j_q}+1}e_{m+\bar{q}} = 0$. Therefore, the ordered set
    \[
    \{x^{2k_{j_q}}e_{m+\bar{q}}, \ x^{2k_{j_q}-1}e_{m+\bar{q}}, \ \ldots \ , \ xe_{m+\bar{q}}, \ e_{m+\bar{q}}\}
    \]
    forms a super Jordan block of size $2k_{j_q}+1$, which is odd.
    Since $\{J_1, \ldots, J_t\} = \{J_{i_1}, \ldots, J_{i_{r_1}}\} \;\sqcup\; \{J_{j_1}, \ldots, J_{j_{r_2}}\}$, we obtain
    \[
    (k_{i_1} + \cdots + k_{i_{r_1}}) + (k_{j_1} + \cdots + k_{j_{r_2}}) = r.
    \]
    Therefore, the total sum of the sizes of the super Jordan blocks constructed above is
    \[
    (2k_{i_1}+1) + \cdots + (2k_{i_{r_1}}+1) + (2k_{j_1}+1) + \cdots + (2k_{j_{r_2}}+1) = (r+r_1) + (r+r_2).
    \]
    Note that each set of individual standard unit vectors
    \[
    \{e_{r+r_1+u}\}_{u=1}^{m-r-r_1}, \quad \{e_{m+r+r_2+u}\}_{u=1}^{n-r-r_2}
    \]
    forms a super Jordan block of size 1.
    The standard unit vectors forming these super Jordan blocks are mutually disjoint, and the sum of the sizes of all super Jordan blocks is
    \[
    (r+r_1) + (r+r_2) + (m-r-r_1) + (n-r-r_2) = m+n.
    \]
    Therefore, $x$ consists solely of super Jordan blocks of odd size.\\
    
    Conversely, suppose $s \neq 0$ and consider $e_{r+r_1+1}$. Then there is no standard unit vector $e_i$ such that $xe_i = e_{r+r_1+1}$. We have
    \[
    xe_{r+r_1+1} = e_{m+r+r_2+1}, \quad x^2e_{r+r_1+1} = 0.
    \]
    Hence, the set $\{xe_{r+r_1+1}, e_{r+r_1+1}\}$ forms a super Jordan block of size 2, which is even. Therefore, $s = 0$.
    Now, suppose that
    \[
    \{J_1, \ldots, J_t\} \neq \{J_{i_1}, \ldots, J_{i_{r_1}}\} \sqcup \{J_{j_1}, \ldots, J_{j_{r_2}}\}.
    \]
    We have two cases:
    \begin{enumerate}[label=(\roman*), leftmargin=0pt]
        \item There exists $J_{\alpha} \in \{J_1, \ldots, J_t\}$ such that
        \[
        J_{\alpha} \notin \{J_{i_1}, \ldots, J_{i_{r_1}}\} \cup \{J_{j_1}, \ldots, J_{j_{r_2}}\}.
        \]
        Let $l_{\alpha} = k_1+\cdots +k_{\alpha}$, and consider the standard unit vector $e_{m+l_{\alpha}}$. Since $J_{\alpha} \notin \{J_{i_1}, \ldots, J_{i_{r_1}}\}$, there is no standard unit vector $e_i$ such that $xe_i = e_{m+l_{\alpha}}$. We then have a sequence
        \[
        e_{m+l_{\alpha}}, \ xe_{m+l_{\alpha}}, \ x^2e_{m+l_{\alpha}}, \ \ldots \ , \ x^{2k_{\alpha}-1}e_{m+l_{\alpha}},
        \]
        where each term satisfies
        \[x^{2u}e_{m+l_{\alpha}} = e_{m+l_{\alpha}-u}, \quad \text{for} \; u = 1, \ldots, k_{\alpha}-1\]
        and 
        \[x^{2u-1}e_{m+l_{\alpha}} = e_{l_{\alpha}-u+1}, \quad \text{for} \; u = 1, \ldots, k_{\alpha}.\] 
        Note that $x^{2k_{\alpha}}e_{m+l_{\alpha}} = 0$. Therefore, the ordered set
        \[
        \{x^{2k_{\alpha}-1}e_{m+l_{\alpha}},\ x^{2k_{\alpha}-2}e_{m+l_{\alpha}}, \ \ldots \ , \ xe_{m+l_{\alpha}}, \ e_{m+l_{\alpha}}\}
        \]
        forms a super Jordan block of size $2k_{\alpha}$, which is even.
        \item There exists $J_{\beta} \in \{J_1, \ldots, J_t\}$ such that
        \[
        J_{\beta} \in \{J_{i_1}, \ldots, J_{i_{r_1}}\} \cap \{J_{j_1}, \ldots, J_{j_{r_2}}\}.
        \]
        Let $J_{\beta} = J_{i_{\beta_1}} = J_{j_{\beta_2}},$ where $\beta_1 \in \{1, \ldots, r_1 \}$ and $\beta_2 \in \{1, \ldots, r_2\}$. Consider $e_{r+\beta_1}$. Then there is no standard unit vector $e_i$ such that $xe_i = e_{r+\beta_1}$. We then have a sequence
        \[
        e_{r+\beta_1}, \ xe_{r+\beta_1}, \ x^2e_{r+\beta_1}, \ \ldots \ , \ x^{2k_{i_{\beta_1}}+1}e_{r+\beta_1},
        \]
        where each term satisfies 
        \[x^{2u}e_{r+\beta_1} = e_{i_{\beta_1}-u+1}, \quad \text{and} \quad x^{2u-1}e_{r+\beta_1} = e_{m+i_{\beta_1} - u + 1}, \quad \text{for} \; u = 1, \ldots , k_{i_{\beta_1}}.\]
        Note that 
        \[x^{2k_{i_{\beta_1}}+1}e_{r+\beta_1} = e_{m+r+\beta_2}, \quad \text{and} \quad x^{2k_{i_{\beta_1}}+2}e_{r+\beta_1} = 0.\] 
        Therefore, the ordered set
        \[
        \{x^{2k_{i_{\beta_1}}+1}e_{r+\beta_1}, \ x^{2k_{i_{\beta_1}}}e_{r+\beta_1}, \ \ldots \ , \ xe_{r+\beta_1}, \ e_{r+\beta_1}\}
        \]
        forms a super Jordan block of size $2k_{i_{\beta_1}}+2$, which is even.
    \end{enumerate}
    Therefore, we conclude that
    \[
    \{J_1, \ldots, J_t\} = \{J_{i_1}, \ldots, J_{i_{r_1}}\} \;\sqcup\; \{J_{j_1}, \ldots, J_{j_{r_2}}\}.
    \]
    This completes the proof of Lemma \ref{lem;lemma8} and, consequently, that of Theorem \ref{thm;theorem1}.
\end{proof}

\end{document}